\documentclass[10pt,A4]{amsart}
\usepackage{amssymb}
\usepackage{amsfonts}
\usepackage{amstext}
\usepackage{algorithmic}
\usepackage{algorithm}
\usepackage{graphicx}
\usepackage{epstopdf}
\usepackage[all]{xy}

\newtheorem{theorem}{Theorem}

\newtheorem{propos}[theorem]{Proposition}
\newtheorem{corollary}[theorem]{Corollary}

\theoremstyle{definition}

\theoremstyle{remark}

\newcommand{\R}{{\mathbb{R}}}

\newcommand{\C}{{\mathbb{C}}}
\newcommand{\Z}{{\mathbb{Z}}}

\newcommand{\CC}{{\mathcal{C}}}
\newcommand{\CE}{{\mathcal{E}}}
\newcommand{\CF}{{\mathcal{F}}}

\newcommand{\CM}{{\mathcal{M}}}

\newcommand{\CQ}{{\mathcal{Q}}}
\newcommand{\CR}{{\mathcal{R}}}

\newcommand{\cv}{{\mathfrak{v}}}
\newcommand{\cu}{{\mathfrak{u}}}

\newcommand{\und}{\underline}

\newcommand{\rlbicross}{{\triangleright\!\!\!\blacktriangleleft}}

\renewcommand{\ker}{{\rm{ker}}}

\newcommand{\trace}{{\rm trace}}
\newcommand{\<}{{\langle}}
\renewcommand{\>}{{\rangle}}
\newcommand{\ev}{{\rm ev}}
\newcommand{\coev}{{\rm coev}}
\newcommand{\Ad}{{\rm Ad}}

\newcommand{\tens}{\otimes}
\newcommand{\id}{{\rm id}}
\newcommand{\bo}{{}^{(1)}}
\newcommand{\bt}{{}^{(2)}}
\renewcommand{\o}{{}_{(1)}}
\renewcommand{\t}{{}_{(2)}}
\renewcommand{\th}{{}_{(3)}}
\newcommand{\fo}{{}_{(4)}}
\newcommand{\fiv}{{}_{(5)}}
\newcommand{\six}{{}_{(6)}}
\newcommand{\extd}{{\rm d}}
\newcommand{\del}{{\partial}}
\newcommand{\eps}{\epsilon}
\newcommand{\tr}{{\rm Tr}}

\newcommand{\ra}{{\triangleleft}}
\begin{document}

\title{$q$-Fuzzy spheres and quantum differentials on $B_q[SU_2]$ and $U_q(su_2)$}

\author{Shahn Majid}
\address{Queen Mary University of London\\
School of Mathematics, Mile End Rd, London E1 4NS, UK}

\email{s.majid@qmul.ac.uk}

\date{29th December 2008 -- revised March 2010}

\begin{abstract} Whereas the classical sphere $\C P^1$ can be defined as the coordinate algebra generated by the matrix entries of a projector $e$ with $\trace(e)=1$, the fuzzy-sphere is defined in the same way by $\trace(e)=1+\lambda$. We show that the standard $q$-sphere is similarly defined by $\trace_q(e)=1$ and the Podle\'s 2-spheres by $\trace_q(e)=1+\lambda$, thereby giving a unified point of view in which the 2-parameter Podle\'s spheres are  $q$-fuzzy spheres.  We show further that  they arise geometrically as `constant time slices' of the unit hyperboloid in $q$-Minkowski space viewed as the braided group $B_q[SU_2]$. Their localisations are then isomorphic to quotients  of $U_q(su_2)$ at fixed values of the $q$-Casimir precisely  $q$-deforming the fuzzy case. We use transmutation and twisting theory to introduce a   $C_q[G_\C]$-covariant calculus on general  $B_q[G]$ and $U_q(g)$, and use $\Omega(B_q[SU_2])$ to provide a unified point of view on the 3D calculi on fuzzy and Podle\'s spheres. To complete the picture we show how the covariant calculus on the 3D bicrossproduct spacetime arises from $\Omega(C_q[SU_2])$ prior to twisting. \end{abstract}
\maketitle 

\section{Introduction}

$q$-deformed geometries were extensively studied in the late 1980s and early 1990s but have recently acquired a new lease of life as effective theories in quantum gravity. In (Euclidean) 3D quantum gravity with cosmological constant  $C_q[SU_2]$ appears as the coordinate algebra of the `frame rotations quantum group' and $U_q(su_2)$ appears as the `model quantum spacetime' coordinate algebra. We refer to \cite{MaSch} for recent work and the physics background and references. In the present paper we address two fundamental problems outstanding from the previous era.

The first, which is our starting point concerns the mysterious role of non-standard quantum spheres.  In \cite{Pod}, Podle\'s classified  $C^*$-algebras that could reasonably be viewed as deformed `spheres' and which were covariant under the quantum group $C_q[SU_2]$. He found a 2-parameter worth of algebras which can be presented as having generators $x,z,z^*$ and relations 
\begin{equation}\label{podles}
zx=q^2xz,\quad zz^*=(s^2+q^2x)(1-q^2x),\quad z^* z=(s^2+x)(1-x)\end{equation}
where $q^2\ne 0$ and $s^2$ are the real parameters.  Whereas the case $s=0$ is well-studied as the standard $q$-sphere that arises naturally in quantum group methods as the $U(1)$-invariant subalgebra of 
$C_q[SU_2]$ (see \cite{Ma:rieq} for a recent study), the case of nonzero $s$ has remained enigmatic as to how it should be fully understood in $q$-geometry. We do not attempt to list all the literature on these nonstandard `Podle\'s spheres' but we note for example connections to $q$-special functions\cite{NM} and an algebraic construction in terms of coideals of twisted primitives, see eg. \cite{DK}. 

We will show how non-standard Podle\'s spheres arise as $q$-deformed fuzzy spheres essentially as quotients of $U_q(su_2)$ by setting the $q$-casimir to a constant. This $q$-deforms the role of $U(su_2)$ as quantisation of $su_2^*$ (`fuzzy $\R^3$') and the normal `fuzzy sphere' as a quotient of it quantizing a coadjoint orbit\cite{BatMa:non}.  Note that the term `fuzzy sphere' is also used more narrowly in the physics literature for finite-dimensional matrix algebras (these arise at certain discrete radii). Similarly when $s^2=-q^{2n}$ for a positive integer $n$, the Podle\'s sphere is a finite-dimensional matrix algebra and this has been called `$q$-fuzzy sphere', e.g. \cite{GMS}. Our results should not be confused with such notations but are broadly compatible with them. Also note that in the current physical picture of 3D quantum gravity the `fuzzy' aspect comes from quantum gravity while the further $q$-deformation is the introduction of a cosmological constant. From a mathematical point of view it is also interesting that the Podle\'s spheres are {\em both} subalgebras of $C_q[SU_2]$ and, essentially, quotients of its dual, allowing them in different limits to interpolate between the classical sphere and the fuzzy sphere. Our results, in Section~3, come from a unified approach to quantum spheres in Section~2 provided by a systematic `braided trace' construction. 

The second problem is the quantum differential geometry of algebras like $U_q(su_2)$ viewed `up-side-down' as noncommutative spaces. We provide a natural calculus here in Section~4 which $q$-deforms the quantum differential calculus on $U(su_2)$  previously introduced in \cite{BatMa:non}. It induces a natural quantum differential calculus on the Podle\'s or $q$-fuzzy sphere $q$-deforming the calculus on the usual fuzzy sphere. The $q$-geometry of interest here is more precisely the braided group $B_q[SU_2]$. This is the 3D unit $q$-hyperboloid in $q$-Minkowksi space  and obtained from $C_q[SU_2]$ by a covariantisation process of transmutation, see \cite[Chapter 10]{Ma:book}. As an algebra and for generic $q$, $U_q(su_2)$ is a localisation of $B_q[SU_2]$ but it is the latter which appears more natural in the noncommutative geometry.  The non-standard Podle\'s sphere appears as a `constant time' slice of this $q$-hyperboloid. Moreover, we use braided methods to provide a natural differential graded algebra  $\Omega(B_q[SU_2])$ by transmutation of the standard 4D calculus on $C_q[SU_2]$. The general Hopf algebra theory behind this is in the Appendix and provides a natural quantum differential calculus on all $B_q[G]$ associated to semisimple Lie algebras, and hence on $U_q(g)$ as their localisations. The calculus is constructed for the transmutation of any coquasitriangular Hopf algebra and can also be understood as a comodule algebra twist. As a result the calculus is in fact covariant under the complexification $C_q[G_\C]=C_q[G]\bowtie_\CR C_q[G]$. The twist result motivates a parallel view of the calculus on $C_q[SU_2]$ as a deformation of the calculus on the 3D bicrossproduct spacetime, a picture completed in Section~5.

\subsubsection*{Acknowledgements} The bicrossproduct  result in Section 5 was presented at the ICMS conference on noncommutative deformations of special relativity,  Edinburgh, 2008, and the twisting result in the Appendix at the QGQG conference, Corfu 2009.

\section{Uniform projector construction of quantum spheres}

{\bf 2.1} We first recall the `quantum logic' point of view on $\C P^1$ used in \cite{BraMa:qua}. Thus, specifying a line in $\C^2$ is the same thing as specifying a matrix $e$ of size $2\times 2$ and obeying 
\begin{equation}\label{projrel} e^2=e,\quad e^\dagger=e\end{equation}
\begin{equation}\label{tr1}\trace(e)=1.\end{equation}
As the matrix is hermitian its eigenvalues are real. As it is a projector its eigenvalues are 0,1 and as
trace is 1, its image is a 1-dimensional subspace of $\C^2$ as the eigenspace of eigenvalue 1. We now use such matrices to `coordinatise' $\C P^1$, i.e. we regard its entries as generators of the coordinate $*$-algebra $A$ and (\ref{projrel})-(\ref{tr1}) as its defining relations.

It is an easy exercise to write $e=\begin{pmatrix} 1-a & b \\ b^* & a\end{pmatrix}$ for self-adjoint generator $a$ and complex generator $b$. The form of $e$ solves the trace and hermitian conditions and the remaining projector relations become
\[ ba=ab,\quad bb^*=b^* b=a(1-a)\]
which indeed describes a sphere of radius $1/2$ if we write $b=-x_1+\imath x_2$ and $a=x_3+{1\over 2}$ (then the last relation here is $\sum_i x_i^2={1\over 4}$.) If $\sigma_i$ are the usual Pauli matrices then $e={1\over 2}-\sigma\cdot x$ in this Cartesian basis.

In this description the projector $e$ also provides the tautological (monopole) bundle on the sphere. The space of sections of this bundle is the projective module 
\begin{equation}\label{projmod}\CE=\{e\begin{pmatrix}f \\ g\end{pmatrix}\ |\ f,g\in A\}.\end{equation}
An equivalence class $[e]$ of this `Bott projector' defines an element of the $K$-theory of the sphere. 

{\bf 2.2}  The fuzzy sphere is the standard quantisation of a coadjoint orbit in $su_2^*$ with its Kirillov-Kostant Poisson structure. Thus, the enveloping algebra $U(su_2)$ is regarded as `fuzzy $\R^3$'\cite{BatMa:non} which we write for our purposes as the $*$-algebra 
\begin{equation} [x_i,x_j]=-\imath\lambda\eps_{ijk}x_k,\quad \sum_i x_i^2={1-\lambda^2\over 4}\label{fuzzy},\quad x_i^*=x_i\end{equation}
where $\eps_{123}=1$ and $\eps_{ijk}$ is totally antisymmetric. Note that in the conventions here are for a fuzzy sphere of radius $1/2$ and our real parameter $\lambda$ is dimensionless; by suitable rescaling of the generators one can recover the formulae for general radius and the physical $\lambda$ in the physics literature. (In suitable units the radius is often taken to be discrete so that the algebras are matrix blocks but this is not required in the noncommutative geometry as explained in \cite{BatMa:non}). It is easy to check using the usual properties $\sigma_i\sigma_j=\delta_{ij}+\imath\eps_{ijk}\sigma_k$ of the Pauli matrices that the  monopole projector $e$ is now deformed to $e={1+\lambda\over 2}-\sigma\cdot x$.

It is also easy in the point of view of \cite{BraMa:qua} to run the calculation that $e$ is a projector  backwards, i.e. we  write $e=\begin{pmatrix} 1+\lambda-a & b \\ b^* & a\end{pmatrix}$ so that 
\begin{equation}\label{trlambda} \trace(e)=1+\lambda\end{equation}
as a deformation of $\C P^1$, where we fix $\lambda$ to be some real number. We require $a^*=a$ to maintain the hermitian condition and the remaining projector relation in  (\ref{projrel}) becomes
\[ [a,b]=\lambda b,\quad [b,b^*]=2\lambda(a-{1+\lambda\over 2}),\quad b^*b=a(1-a).\]
The first two are $\lambda$-deformed commutation relations while the last enforces the `sphere'. If we write $a=x_3+ (1+\lambda)/2$ and $b=-x_1+\imath x_2$ we obtain the $*$-algebra (\ref{fuzzy}) and the $\lambda$-monopole projector.

{\bf 2.3} Let $V$ be a rigid object in a $k$-linear braided category ($k$ a field). Rigid here means in practice finite-dimensional and is defined as the existence of evaluation and coevaluation morphisms $V^*\tens V\to k$ and $k\to V\tens V^*$ defined in the obvious way as $\<f^a,e_b\>=\delta^a_b$ and $1\mapsto \sum_a e_a\tens f^a$ for a basis and dual basis of $V$. In a standard diagrammatic representation one has a canonical braided trace defined by
\[ \underline{\trace}(\phi)=\ev\circ(\id\tens\phi)\circ\Psi_{V,V^*}\circ\coev
,\quad\forall \phi:V\to V\]
where $\Psi$ is the braiding that exists between any two objects in the category. 

We work over $k=\C$. The quantum group $C_q[SU_2]$ defines a braided category as its category of comodules (in other words, representations of the the quantum `group') and for the spin 1/2 representation 
\[\underline{ \trace}\begin{pmatrix}a&b\\ c&d\end{pmatrix}=a+q^2d=\trace_q\begin{pmatrix}a&b\\ c&d\end{pmatrix}\]
in a certain normalisation (which we choose for convenience in what follows). The right hand side here is called the $q$-deformed trace and we see how it arises from the braided category\cite{Ma:book}. The matrix here  is viewed as an operator $\C^2\to \C^2$ with values which could be in some other space (in our case in a coordinate $*$-algebra) considered as `bosonic'. If $a,b,c,d$ transform in the same way under $C_q[SU_2]$ as a quantum or braided matrix (under the quantum adjoint coaction) then the $q$-trace is invariant. Hence any relation defined through the $q$-trace is also covariant under this quantum group.

In particular, we look at a matrix of generators $e=\begin{pmatrix} 1-q^{2}a & b \\ b^* & a\end{pmatrix}$ which has
\[ \trace_q(e)=1.\]
We suppose that $q$ is real and $a=a^*$ as before, so that $e^\dagger=e$. The remaining projector relation in (\ref{projrel}) becomes
\[ ba=q^2ab,\quad bb^*=q^4b^*b+aq^2(1-q^2),\quad b^*b=a(1-a)\]
which are the defining relations of the standard $q$-sphere. The first two are $q$-deformations of the commutativity relations while the last enforces the `sphere'. In terms of generators $b_\pm,b_0$ in \cite{Ma:rieq} (as arising out of the $q$-monopole quantum principal bundle) the conversion is
\[ a=-q^{-1}b_0,\quad b=b_+,\quad b^*=-qb_-.\]
The projector $e$ then becomes the one found for the $q$-monopole in \cite{HajMa:pro} up to choice of conventions. This observation was mentioned in \cite{BraMa:qua}.

{\bf 2.4} Clearly we could assign to $\trace_q(e)$ some other real number and retain invariance under $C_q[SU_2]$. Thus the above point of view suggests, as a $q$-deformation of the fuzzy sphere, to require
a projector of the form $e=\begin{pmatrix} 1+\lambda-q^2a & b \\ b^* & a\end{pmatrix}$ so that
\[ \trace_q(e)=1+\lambda\]
where $\lambda$ is a second parameter, also real. We again have $a^*=a$ to retain the hermitian condition. Then the remaining projector relations in (\ref{projrel}) become
\begin{equation}\label{qfuzzy} q^2ab-ba=\lambda b,\quad bb^*=(q^2a-\lambda)(1+\lambda-q^2a),\quad b^*b=a(1-a).\end{equation}
\begin{propos} The $q$-fuzzy sphere $*$-algebra defined by generators $a=a^*, b,b^*$ and real parameters $q^2,\lambda$ (with $q$ invertible) and relations (\ref{qfuzzy}) is
\begin{enumerate}
\item  isomorphic to the fuzzy sphere if $q^2=1$.
\item isomorphic to the  standard $q$-sphere if $\lambda=0,q^2-1$. 
\item isomorphic to the nonstandard Podle\'s sphere if $q^2\ne 1$ and $\lambda\ne 0, q^2-1$. In this case 
\[ s^2={\lambda\over q^2-1-\lambda}\]
(if we want $s$ to be real then we require $\lambda$ to lie between 0 and $q^2-1$).  
\item invariant under $\lambda\mapsto q^2-1-\lambda$ up to isomorphism $b\mapsto b, a\mapsto 1-a$.
\end{enumerate}
\end{propos}
\proof The cases $q^2=1$ in part (1) and $\lambda=0$ in part (2) are clear after minor rearrangements. If $\lambda=q^2-1$ also in part (2)  then the result is again a standard $q$-sphere in terms of $b,b^*$ and a new variable $a'=1-a$ in the role of $a$. For part (3) we first change variables to $a=x'+\lambda'$ where $\lambda'=\lambda/(q^2-1)$. The offset here turns the relations in (\ref{qfuzzy}) to
\[ bx'=q^2x'b, \quad bb^*=(q^2x'+\lambda')(1-\lambda'-q^2x'),\quad b^*\ b=(x'+\lambda')(1-\lambda'-x')\]
which starts to resemble (\ref{podles}). We now consider $x'=x\mu$ and $b=z\nu$ for real parameters $\mu,\nu$. Comparing with (\ref{podles}) we require 
\[ \mu^2=\nu^2,\quad{ \lambda'(1-\lambda')\over\nu^2}=s^2,\quad 1-s^2={\mu(1-2\lambda')\over\nu^2}\]
which we take as a definition of $\nu$ (say $\nu=\pm\mu$), of $s^2$ and an equation for $\mu$ respectively. The latter has solutions $\mu=1-\lambda',-\lambda'$ and we take the first with $\nu=\mu$, say. Then $s^2=\lambda'/(1-\lambda')$ which works out as stated. The chosen algebra isomorphism is given by
\[ b=(1-\lambda')z,\quad a=x(1-\lambda')+\lambda',\quad \lambda'={\lambda\over q^2-1}.\]
The other choice of $\mu$ and $\nu=-\mu$ (for convenience) gives 
\[ b=\lambda'z,\quad a=\lambda'(1-x),\quad s^2=(1-\lambda')/\lambda'\]
as another isomorphism with (\ref{podles}), but with inverse $s$ to the previous choice (hence the Podle\'s sphere is invariant under inversion of $s$ up to rescaling of generators).  This leads to part (4). In terms of the $q$-fuzzy sphere the invariance appears as stated and one can verify directly that it applies in all cases. It corresponds to inversion of $s^2$ but usually one is interested in the sector $q\in(-1,1)$ and $s^2\in [0,1]$, so we do not see this, or in the discrete series $s^2=-q^{2n}$ where it corresponds to inversion of $q$.
\endproof

 By construction, the $q$-fuzzy sphere comes with a tautological projective module, namely cross-sections of  the $q$-fuzzy monopole as defined by (\ref{projmod}). This necessarily agrees, up to conventions, with the previous special cases and with the projector for the nonstandard Podle\'s sphere found in \cite{BrzMa:geo}.

{\bf 2.5} Clearly, this construction works for any rigid object of a $\C$-linear braided category. Let $R\in M_2\tens M_2$ obey the braid or `quantum Yang-Baxter' equations. Here $V=\C^2$ with basis $\{e_i|i=1,2\}$ and dual basis $\{f^i\}$ and \cite[Ex 9.3.12]{Ma:book}
\[ 
\Psi_{V,V}(e_i\tens e_j)=e_b\tens e_a R^a{}_i{}^b{}_j,\quad \Psi_{V,V^*}(e_i\tens f^j)=\tilde R^a{}_i{}^j{}_bf^b\tens e_a\]
where $\tilde R$ is the `second inverse' of $R$. We then find the 
braided trace of a projection $e=(e^i{}_j)$ as
\[ \underline{\trace}(e)=\trace(eu);\quad u^i{}_j=\tilde R^a{}_j{}^i{}_a.\]
Hence the relations of the general braided-fuzzy-sphere are
\begin{equation}\label{brasphere} \trace(eu)=1+\lambda,\quad e^2=e,\quad e^\dagger=e.\end{equation}
\begin{propos} We obtain a $*$-algebra with relations (\ref{brasphere}) if $\lambda$ is real and $R$ is of `real type' in the sense \cite[Defn. 4.2.15]{Ma:book}
\[ \overline{R^i{}_j{}^k{}_l}=R^l{}_k{}^j{}_i.\]
\end{propos}
\proof The second inverse is defined by $\tilde R=((R^{t_2})^{-1})^{t_2}$ where $t_2$ denoted transpose in the second copy of $M_2$ (the last two indices of $R$). The reality condition implies that the matrix $u$ is hermitian and this in turn implies that the quantity $\trace(eu)$ is self-adjoint in the $*$-algebra. It therefore makes sense to assign to it a real value as stated. The $R$-matrix also leads to a quantum group with coquasitriangular structure \cite[Prop. 4.2.2]{Ma:book} by construciton $\trace(eu)$ will be invariant under its coaction where $e$ transforms in the adjoint coaction (by `conjugation of the matrix generators), as one may directly check. As the other relations are likewise covariant (here again the reality property of $R$ is used) one has covariance of the resulting braided sphere, i.e. it lives in the braided category of comodules of the quantum group. \endproof

A great many $M_2\tens M_2$ bi-invertible solutions $R$ of real type are known and hence we obtain $R$-fuzzy spheres  and $R$-monopole bundles for all of them. They include 2-parameter deformations and the symmetric (not braided) Jordanian solution for real parameter values, among others 
\cite{Ma:book}. Note that the normalisation of $R$ affects the value of the parameter $\lambda$ but for a fixed normalisation (a fixed braided category) the parameter $\lambda$ represents an additional freedom. 

\section{$q$-fuzzy sphere as a quotient of $U_q(su_2)$ and of the $q$-hyperboloid}

Just as one can write the fuzzy sphere as a quotient of $U(su_2)$ by a specified value of the Casimir, in view of the above derivation we would expect to be able to do the same now. This turns out to be possible but only provided we allow $s^2$ to be negative  (when $q$ is real as we assume throughout). 

{\bf 3.1} We first recall the standard structure of the quantum group $U_q(su_2)$ as having generators $x_\pm, K,K^{-1}$  (one can write $K=q^{H/2}$ as a suggestive notation), with relations
\[ Kx_\pm K^{-1}=q^{\pm 1} x_\pm,\quad [x_+,x_-]={K^2-K^{-2}\over q-q^{-1}}, \quad K^*=K,\quad x_\pm^*=x_\mp\] 
where the $*$-structure is for real $q$. The algebra has a $q$-deformed quadratic Casimir
\[ c_q=K^2q^{-1}+qK^{-2}+x_+x_-(q-q^{-1})^2.\]
We refer to \cite{Ma:book} for further details in these conventions.

\begin{propos} 
Let $q\ne 1$ and $s=\imath t$ where $t$ is real. The `patch' in the Podle\'s $2$-sphere where we adjoin $x^{-1}$ is isomorphic to the algebra  $U_q(su_2)$ modulo the relation $c_q=t+t^{-1} $.
\end{propos}
\proof We set $x=\mu K^2$, $z=\nu Kx_-$ and hence $z^*=\nu x_+K$. We take $\mu,\nu$ real (the latter for simplicity). Then clearly the $zx=q^2xz$ relation holds. We likewise compute
\[ z^*z=\nu^2x_+K^2x_-={\nu^2 q^{-2}\over (q-q^{-1})^2}(c_q-K^2q^{-1}-qK^{-2})K^2={\nu^2 q^{-2}\over (q-q^{-1})^2}(c_q{x\over\mu}-{x^2\over \mu^2 q}-q).\]
Comparing with the last of (\ref{podles}) we need
\[ \mu^2 q^2=t^2,\quad 1+t^2=c_q \mu q,\quad \nu^2=qt^2(q-q^{-1})^2. \]
We solve this with $\mu=q^{-1}t$ and $\nu=q^{{1\over 2}}t(q-q^{-1})$ (say). The choice of signs of the square roots is not fixed by present considerations (see later). We then verify the middle relation in (\ref{podles}) and indeed that the two algebras are then isomorphic as stated. \endproof

{\bf 3.2} For a full geometrical picture we need the $*$-algebra $B_q[SU_2]$. This is defined as the transmutation of $C_q[SU_2]$ to a braided Hopf algebra in the braided category of $U_q(su_2)$ (left) modules (a braided group), see \cite[Chap. 10]{Ma:book}. It is a covariant version of $C_q[SU_2]$ and is a quotient by a braided-determinant of a braided bialgebra $B_q[M_2]$ of $2\times 2$ braided Hermitian matrices. The latter is a natural candidate for $q$-Minkowski space and has a matrix of generators $u=\begin{pmatrix}\alpha & \beta \cr \beta^* &\delta\end{pmatrix}$ where $\alpha^*=\alpha$ and $\delta^*=\delta$, with
\[ \beta\alpha=q^2\alpha\beta,\quad \delta\alpha=\alpha\delta,\quad [\beta,\beta^*]=(1-q^{-2})\alpha(\delta-\alpha),\quad [\delta,\beta]=(1-q^{-2})\alpha\beta.\]
The additional braided determinant relation for $B_q[SU_2]$ is
\[ \alpha\delta-q^2\beta^*\beta=1\]
which from the point of view of $q$-geometry makes this a $q$-hyperboloid or `mass-shell' of unit Lorentzian distance from the origin. Both algebras have a natural central element $\tr_q(u)=q^{-1}\alpha+q\delta$ which is, up to a normalisation, the braided trace used before. It is the time coordinate in the $q$-Minkowski space in usual cartesian coordinates. 

\begin{propos} Let $s=\imath t$ where $t\ne 0$ is real. Then the Podle\'s $2$-sphere is isomorphic to the `time slice' $\tr_q(u)=t+t^{-1}$ of $B_q[SU_2]$. Moreover, when $t^2\ne 1$ the projector $e$  and the parameter $\lambda$ in Section~2 for the general Podle\'s sphere takes the form
\[ e={1-qt u\over 1-t^2},\quad \lambda=t^2{1-q^2\over 1-t^2}\]
\end{propos}\proof Let us note by way of explanation that $B_q[SU_2]$ also has localisation if we allow $\alpha$ invertible which is isomorphic to the algebra $U_q(su_2)$ (this is part of the self-duality of these braided groups in a formal power-series setting). Namely \cite{Ma:book}
\[ u=\begin{pmatrix} K^2 & q^{-{1\over 2}}(q-q^{-1})Kx_-\cr q^{-{1\over 2}}(q-q^{-1})x_+K &\phantom{HHHH} K^{-2}+q^{-1}(q-q^{-2})^2x_+x_-\end{pmatrix}.\]
Replacing $K^{-2}$ in favour of $c_q$ and comparing with our previous proposition, we have
\[ u={1\over qt}\begin{pmatrix}q^2 x & z \cr z^* &\phantom{HH} t^2+1-x\end{pmatrix}\]
which we then verify to hold globally (not requiring $x$ invertible). Comparing further with
\[ e= (1-\lambda')\begin{pmatrix} 1-q^2x & -z \cr -z^* & x+{\lambda'\over 1-\lambda'}\end{pmatrix}\]
for the projector in Section~1, we arrive at the result stated provided $t^2\ne 1$. Here $\lambda'/(\lambda'-1)=t^2$ and $\lambda=\lambda'(q^2-1)$ which  we may solve in terms of $t$. Note that 
\[ u^2=-{1\over q^2}+ {t+t^{-1}\over q}u\]
from which one may directly verify that $e$ is a projector and that a projector built from a linear combination of $1$ and $u$ is only possible if $t^2\ne 1$ to give the above $e$ or its complement.   \endproof

\section{Quantum differential calculus on $B_q[SU_2]$ and $U_q(su_2)$}

{\bf 4.1} A first order differential calculus over an algebra such $C_q[SU_2]$,  means $\Omega^1$ a bimodule over the algebra and $\extd:C_q[SU_2]\to \Omega^1$ obeying the Leibniz rule. In addition the span over the algebra of the image of $\extd$ is all of $\Omega^1$ and, at least for generic $q$, kernel of $\extd$ is the linear span of the identity $1$ of the algebra (an optional connectedness condition). Such a notion has been extensively studied for  standard quantum groups and those $\Omega^1$ which are both left and right translation covariant (`bicovariant') under the quantum group coacting on itself have been classified. For $C_q[SU_2]$ and generic $q$ there is one smallest bicovariant calculus, of dimension 4. It was first found in \cite{Wor:dif}. For any bicovariant calculus on  a quantum group Woronowicz showed how to extend $\Omega^1$ to an exterior algebra $\Omega$ with $\extd$ extended as a graded-derivation. It is known (`Brzezinski's theorem') that the latter is a super or $\Z_2$-graded Hopf algebra.

{\bf 4.2} Although $q$-Minkowski space has  a known 4-dimensional differential calculus induced by its additive braided group structure (`braided coaddition')\cite{Ma:book}, this is {\em not} compatible with the $q$-determinant relation and hence does not descend to the  $q$-hyperboloid $B_q[SU_2]$. 

Instead we use that the latter is a multiplicative braided group obtained by a theory of `transmutation' from $C_q[SU_2]$. This theory works for all the standard quantum groups $C_q[G]$ (in fact for all coquasitriangular Hopf algebras). Briefly, we consider the braided category $\CC=\CM^{C_q[G]}$ of $C_q[G]$-comodules and apply braided Tannaka-Krein reconstruction \cite{Ma:bg} to the identity functor from this to itself (essentially). This therefore factors through the braided category $\CC^B$ of $B$-comodules in $\CC$ for a certain braided group $B=B_q[G]$ in the category $\CC$. This `transmutation' process renders a `braided version' of $C_q[G]$ which has the advantage of being fully $C_q[G]$-covariant by virtue of living in the category $\CC$. 

In the Appendix we suppose that $C_q[G]$ (or any other coquasitriangular Hopf algebra)  is equipped with a bicovariant calculus and similarly  apply transmutation to this. Thus we consider the functor $\CF$ induced by a map $\pi$
\[ \CF:\CM^{\Omega(C_q[G])}\to \CC,\quad \pi:\Omega(C_q[G])\to C_q[G],\quad  \pi|_{\Omega^0}=\id,\quad \pi|_{\Omega^i,i>0}=0\]
where by definition $\Omega^0=C_q[G]$. By a super-version of the reconstruction theorem we will obtain a $\Z_2$-graded (super) braided-group $\underline{\Omega}=\Omega(B_q[G])$, i.e. we take the result as a definition of the latter.  It is a differential graded algebra fully $C_q[G]$-covariant by virtue of living in the category $\CC$. The exterior derivative $\extd$ is unchanged on the underlying vector spaces and remains a graded derivation.

The general result  expressed in $R$-matrix form is found to be the following. $\Omega^1(B_q[G])$ is a free module over $B_q[G]$ spanned by $M_n(\C)$, i.e. with basis $\{e_\alpha{}^\beta\}$ (these are identified in the transmutation process with the standard left-invariant 1-forms on $C_q[G]$ for suitable $n$) but with a bimodule structure:
\begin{equation}\label{Rcalc}  R^m{}_\alpha{}^a{}_d R^{-1}{}^\beta{}_n{}^d{}_c e_m{}^n u^c{}_b=u^a{}_c e_m{}^n R^m{}_\alpha{}^c{}_d R^d{}_b{}^\beta{}_n.\end{equation}

in our right-comodule conventions (they may take a more compact form in other conventions). The calculus is inner for generic $q$, 
\begin{equation}\label{theta} \theta=e_\alpha{}^\alpha,\quad \extd =(1-q^{-2})^{-1}[\theta, (\ )].\end{equation}
The element $\theta$ is invariant under the coaction  and has relations unchanged under transmutation. The higher exterior algebra has an unchanged form of relations among the $\{e_\alpha{}^\beta\}$,  i.e the same as for $\Omega(C_q[G])$.

{\bf 4.3} Specifically in the case of $B_q[SU_2]$ we obtain:

\begin{propos}  $B_q[SU_2]$ has a natural 4-dimensional adjoint $C_q[SL_2]$-covariant differential calculus with left basis 1-forms $(e_\alpha{}^\beta)=\begin{pmatrix}e_a& e_b\\ e_c & e_d\end{pmatrix}$ (so $e_c=e_2{}^1)$ and bimodule relations
\[ [e_a,\alpha]_q=[e_a,\beta]_{q^{-1}}=[e_c,\beta]_q=[e_b,\alpha]_{q^{-1}}=[e_b,\gamma]_q=0\]
\[ [e_a,\gamma]_q=\mu \alpha e_b,\quad [e_a,\delta]_{q^{-1}}=\mu \beta e_b+q\mu^2 \alpha e_a,\quad [e_c,\alpha]_q=q^2\mu \beta e_a,\quad [e_b,\beta]_{q^{-1}}=\mu \alpha e_a\]
\[  [e_b,\delta]_q=q^2\mu \gamma e_a,\quad [e_d,\alpha]_{q^{-1}}=\mu \beta e_b,\quad  [e_d,\beta]_q=\mu \alpha e_c+q\mu^2 \beta e_a,\quad [e_d,\gamma ]_{q^{-1}}=\mu (\delta-\alpha)e_b\]
\[ [e_d,\delta]_q=-\mu\beta e_b+q\mu^2(\delta-\alpha) e_a+\mu\gamma e_c,\quad [e_c,\gamma]_{q^{-1}}=\mu(\delta -\alpha)e_a+\mu\alpha e_d+q\mu^2 \beta e_b\]
\[ [e_c,\delta]_{q^{-1}}=\mu(q^2-2)\beta e_a+q^2\mu \beta e_d+q\mu^2\alpha e_c\]
The calculus is inner for $q^2\ne 1$ with $\theta=e_a+e_d$. Here $\mu=1-q^{-2}$ is a shorthand. 
\end{propos}
\begin{proof} We use the standard $R$-matrix in Hecke normalisation with nonzero entries $R^1{}_1{}^1{}_1=R^2{}_2{}^2{}_2=q$, $R^1{}_1{}^2{}_2=R^2{}_2{}^1{}_1=1$ and $R^1{}_2{}^2{}_1=q-q^{-1}$. We allow for an extra factor of $q^{-1}$ on the right hand side of the bimodule relations to convert $R$ to the quantum group normalisation assumed there. \end{proof}
 
 Note that on the generators the exterior derivative has exactly the same form 
 \[ \extd \begin{pmatrix}\alpha\\ \gamma\end{pmatrix}=\mu^{-1}((q-1)\begin{pmatrix}\alpha\\ \gamma\end{pmatrix} (e_a-q^{-1}e_d)+\mu \begin{pmatrix}\beta\\ \delta\end{pmatrix} e_b)\]
 \[ \extd \begin{pmatrix}\beta\\ \delta\end{pmatrix}=\mu^{-1}((q-1)\begin{pmatrix}\beta\\ \delta\end{pmatrix} (e_d-q^{-1}e_a)+\mu \begin{pmatrix}\alpha\\ \gamma\end{pmatrix} e_c+q\mu^2\begin{pmatrix}\beta\\ \delta\end{pmatrix}e_a)\]
 as for $C_q[SU_2]$ on its generators. However, the algebras are different. One may then verify from these formulae, as a cross-check, that $\extd (\alpha\delta-q^2\gamma\beta)=0$ when computed using the above, the Leibniz rule and the stated bimodule relations. The relations among the left-invariant basis 1-forms are likewise the same as for $C_q[SU_2]$ and hence take the same form in our conventions: $e_a,e_b,e_c$ behave as usual Grassmann variables and
 \[ e_ae_d+e_de_a+\mu e_ce_b=0,\quad e_de_c+q^2e_ce_d+\mu e_ae_c=0\]
 \[ e_be_d+q^2e_de_b+\mu e_be_a=0,\quad e_d^2=\mu e_ce_b.\]
One may verify as a cross-check that the bimodule relations of the proposition are consistent with these 1-form relations, for example when used to compute $e_c^2\delta=0$. Together with the bimodule relations they generate the entire transmutated exterior algebra of $B_q[SU_2]$ with $\extd$ a graded-derivation given by graded commutator with $\theta$.

\begin{corollary} The above calculus localises to a 4-dimensional calculus on $U_q(su_2)$ with 
\[ e_aK=q^{1\over 2}Ke_a,\quad e_b K=q^{-{1\over 2}}Ke_b,\quad e_a x_-=q^{-{3\over 2}}x_- e_a,\quad [e_a,x_+]_{q^{-{1\over 2}}}=Ke_b\]
\[ [e_c,K]_{q^{1\over 2}}=\mu(q-1)x_- e_a,\quad [e_d,K]_{q^{-{1\over 2}}}=\mu(1-q^{-1})x_- e_b,\quad [e_b,x_-]_{q^{-{1\over 2}}}=\mu K e_a\]
\[ [e_c,x_-]_{q^{1\over 2}}=\mu q^{-2}(1-q)K^{-1}x_-^2e_a,\quad [e_d,x_-]_{q^{3\over 2}}=q^{3\over 2}\mu^2x_-e_a+Ke_c+\mu (q^{-1}-1)K^{-1}x_-^2e_b\]
\[ [e_a,x_+]_{q^{1\over 2}}=Ke_b,\quad e_bx_+=q^{3\over 2}x_+e_b,\quad  [e_d,x_+]_{q^{-{1\over 2}}}=\mu K^{-1}(qx_-x_+-x_+x_-)e_b\]
\[ [e_c,x_+]_{q^{-{3\over 2}}}=\mu K e_d+\mu q^{1\over 2}(1-q^{-1})x_-e_b+\mu K^{-1}(x_-x_+-q^{-1}x_+x_-)e_a\]
We define the exterior derivative by the inner form $\extd=[\theta,(\  )]$. The calculus is covariant under the adjoint action of $U_q(su_2)$ and recovers the calculus of \cite{BatMa:non} on `fuzzy $\R^3$'  as $q\to 1$.
\end{corollary}
\begin{proof} The first two relations follow easily as the square root of the relations with $K^2=\alpha$. That the other relations similarly factorise is not obvious but can be done as stated; one may verify from the stated relations that these imply the desired  $B_q[SU_2]$  bimodule relations.
The exterior derivative is now defined {\em without} the $\mu^{-1}$ normalisation in order to have a limit as $q\to 1$ in the new generators. Then one can compute relations such as
\[ \extd K.K=(1+\lambda)K\extd K+\lambda K^2\theta,\quad \lambda=q^{1\over 2}(1-q^{-{1\over 2}})^2\]
and more complicated bimodule relations with $\extd x_\pm$.  The displayed relation has the same form as one may compute for an exponentiated generator of fuzzy $\R^3$ using formulae in \cite{Ma:time} and an appropriate matching of parameters. Similarly for the other relations. One can rework this to formally derive the fuzzy $\R^3$ calculus relations in a similar manner to our treatment of the less familiar bicrossproduct calculus in the next section. The covariance under the adjoint action is by evaluation against the right adjoint coacition and covariance properties of the quantum killing form used in the identification of a localisation of $B_q[SU_2]$ with $U_q(su_2)$, see \cite{Ma:book}.   \end{proof}

There is actually a larger covariance of this calculus, namely under an action of the quantum double $D(U_q(su_2))$. This is explained in the next section and proven in the Appendix.
 
\begin{propos} The above calculus on $B_q[SU_2]$ descends to a $C_q[SU_2]$-covariant 3-dimensional calculus on the $q$-fuzzy-sphere with the additional relation
\[ (t+t^{-1})\theta=q^{-1}(1+q^{-1})(\alpha e_d+\delta e_a- q^{-1}\beta e_b - q \gamma e_c).\]
\end{propos}
\begin{proof} We recall that the quotient is $\tr_q(u)=t+t^{-1}$, a constant, and that here $\tr_q(u)=q^{-1}\alpha+q\delta$. We compute  $\extd \tr_q(u)=0$ using the explicit formulae for $\extd$ above, and find for $q\ne 1$ that  this becomes $\tr_q(u)\theta$ equal to the right hand side of the stated expression, giving the result.  With this additional relation $\Omega^1$ pulled back to the $q$-fuzzy sphere is no longer a free module. Note that all constructions are $C_q[SU_2]$ covariant under the adjoint coaction on $B_q[SU_2]$ which descends to a coaction on the $q$-fuzzy sphere as $\tr_q(u)$ is invariant. Note also that the constraint was derived for $q\ne 1$ but the resulting calculus makes sense also at $q=1$ as an extension of the classical calculus on the time slice of the hyperboloid.  \end{proof}

Note that covariant calculi on the Podle\'s spheres were studied by other means in \cite{Pod3} and the above is presumably equivalent to the 3-dimensional calculus for the non-standard Podle\'s sphere found there, obtained now from a general R-matrix construction. On the other hand, our methods apply to all the $B_q[G]$ to give a calculus on $U_q(g)$ and hence $U(g)$ for all complex semisimple $g$ and quantisations of their associated coadjoint orbits.

\section{$\Omega(B_q[SU_2])$ as a cotwist and 3D quantum gravity}

{\bf 5.1} The Appendix provides a different route to constructing the above  calculus $\Omega(B_q[G])$, see Proposition~10, which has the merit of exhibiting a larger covariance quantum group $C_q[G_\C]$, where $G_\C$ is the complexification of $G$. We start with the quantum group $C_q[G]^{\rm op}$ equipped with a bicovariant calculus $\Omega(C_q[G]^{\rm op}):=\Omega(C_q[G])^{\rm op}$ obtained from one on $C_q[G]$. We regard the bicovariance of the calculus as covariance under $C_q[G]\tens C_q[G]^{\rm op}$. It is known \cite{Ma:euc} that there is a certain cocycle $F$ on this larger quantum group which Drinfeld-cotwists it into $C_q[G_\C]$ and which in the process cotwists the covariant algebra $C_q[G]^{\rm op}$ into $B_q[G]$. See \cite{Ma:book} for an exposition. One might expect, and this is proven in the appendix, that this process also cotwists $\Omega(C_q[G]^{\rm op})$ into $\Omega(B_q[G])$.

In the case of $G=SU_2$ we regard $C_q[SU_2]^{\rm op}$ as a unit 3-sphere in $q$-Euclidean space with a coaction of $C_q[SU_2]\tens C_q[SU_2]^{\rm op}$ regarded as  a q-deformed covering $C_q[\tilde{SO_4}]$ of $SO_4$, see \cite{Ma:euc}. The cotwist or `quantum Wick rotation' now changes this covariant system to $B_q[SU_2]$ regarded as a unit 3-hyperboloid in $q$-Minkowski space with covariance quantum group $C_q[\tilde{SO_{1,3}}]$, where we regard $(SU_2)_\C=SL(2,\C)$ as a covering of $SO_{1,3}$. The standard 4D bicovariant calculus on $C_q[SU_2]$ adapted to $C_q[SU_2]^{\rm op}$ becomes the calculus on $B_q[SU_2]$ in Proposition~5. Note that these explanations and notations are at an algebraic level that works over any field and should not be confused with $*$-algebra structures on our objects which have to be found separately to fully justify the real forms indicated. This explains the point of view on the $q$-geometry in earlier sections.

{\bf 5.2} As explained in \cite{MaSch}, this cotwisting takes on a meaning in 3D quantum gravity as quantum Born reciprocity, where an interchange of position and momentum in the interpretation of the 3D quantum gravity, and an interchange of the Planck length with the cosmological length scale in the model, is essentially (at an algebraic level) implemented by the cotwist from the $q$-Euclidean  system to the $q$-Minkowski one. Thus the `local picture' or model spacetime in 3D quantum gravity as it naturally emerges consists of particles on $B_q[SU_2]$ with the quantum double or $C_q[\tilde{SO_{1,3}}]$ as isometry quantum group, but its `semidual' or Born-reciprocal model consists of particles on $C_q[SU_2]^{\rm op}$ with $C_q[\tilde{SO_4}]$ isometry. We have explained in Section~4 that the calculus on $B_q[SU_2]$ becomes in a certain limit a calculus on $U(su_2)$ or `fuzzy $\R^3$' and this is known to have a quantum double $D(U(su_2))$ symmetry as the isometry quantum group for 3D quantum gravity without cosmological constant. This emerges as the coaction of  $C_q[\tilde{SO_{1,3}}]$  can be viewed as an action of $D(U_q(su_2))$, which becomes $D(U(su_2))$ in this limit. 

It is also shown in \cite{MaSch} that on the other side of the cotwist or semidualisation we have a different scaling limit in which $C_q[SU_2]^{\rm op}$ becomes the 3D version of the bicrossproduct spacetime\cite{MaRue} and the quantum enveloping algebra corresponding to $C_q[\tilde{SO_4}]$ becomes the bicrossproduct quantum Poincar\'e group $U(su_2)\rlbicross C[SU_2^\star]$  in  \cite{Ma:pla,Ma:book}, where $SU_2^\star$ is a certain solvable group of upper triangular matrices.  Indeed $U_q(su_2)\tens U_q(su_2)^{\rm cop}$ is isomorphic to the bicrossproduct $U_q(su_2)\rlbicross C_q[SU_2^\star]$ and this then becomes the expected bicrossproduct Poincar\'e group in the limit $q\to 1$. This is also why the latter is quasitriangular. We now complete this picture. We work with $C_q[SU_2]$ not its opposite algebra as this boils down to a choice of conventions. It is known that the bicrossproduct spacetimes do not have a quantum Poincar\'e covariant calculus of classical dimensions but this can be remedied with an extra dimension\cite{Sitarz}.

\begin{propos} In the limit whereby $C_q[SU_2]$ becomes the 3D bicrossproduct spacetime, its 4D bicovariant calculus becomes the natural 4D quantum-Poincar\'e-covariant calculus on the bicrossproduct spacetime.
\end{propos}
\proof Taking $C_q[SU_2]$ with its usual matrix of generators $a,b,c,d$ and 4D bicovariant calculus with basis $\{e_a,e_b,e_c,e_d\}$ as above, we work in a `patch' where $a$ is invertible and think of this as an exponentiated generator $q^{z\over \imath\lambda}$ (but we work with $a$ to start with). We write $b={q\mu \over \imath\lambda} q^{-{1\over 2}}x_-$ and $c={q\mu\over \imath \lambda} q^{{1\over 2}}x_+$ in terms of the other two spacetime generators $x_\pm={1\over 2}(x\pm \imath y)$ with $x,y$ self-adjoint when $q$ is real and a real scaling parameter $\lambda$.  The relations in these terms are now $[x,y]=0$ and exponentiated versions $xa=qax$, $ya=qay$ of the bicrossproduct spacetime relations $[x,z]=\imath\lambda x$ and $[y,z]=\imath\lambda y$. Note that $d$ is fixed by the $q$-determinant relations and is not regarded as a generator in this patch.  These matters are all explained in \cite{MaSch} where $x,y,z$ have dimensions of length and the parameter $\lambda$ is the Planck length. We now compute in $\Omega^1(C_q[SU_2])$ using conventions as in \cite{Ma:ric} except that we define the exterior derivative on functions as $\extd=(\imath\lambda)^{-1}[\theta,\ ]$ where $\theta=\imath(e_a+e_d)$, i.e.  {\em without} the $\mu^{-1}$ factor so as to have a limit on $z,x_\pm$ as $q\to 1$. We use conventions so that $\theta$ is self-adjoint when $q$ is real. Note also that  $z^*=z$ in the limit. The nonzero bimodule commutation relations in this limit become
\[ [e_a,z]=[e_b,x_-]=[e_c,x_+]=\imath\lambda e_a,\quad [e_d,(z,x_-,x_+)]=\imath\lambda(-e_d,e_c,e_b)\]
where the relations involving $z$ actually arise in an exponentiated form, $a^{-1}e_aa=qe_a+O(\mu)$, $a^{-1}e_da=q^{-1}e_d+O(\mu)$ prior to taking the limit, which we interpret as shown. Since we are working here algebraically the correct statement is that there is a calculus as shown on the algebra generated by $z,x_\pm$ or $z,x,y$ and commutation relations which extend in the expected way to any completion allowing exponentials. In our calculus we compute $\extd z=\imath(e_a-e_d),\quad \extd x_-=\imath e_c,\quad \extd x_+=\imath e_b$. Moving to the self-adjoint generators $x_i=x,y$ we have  the equivalent presentation
\begin{equation}\label{bicrosscalc} [\extd x_i,x_j]=\imath\lambda \delta_{ij}(\theta+ \extd z),\quad[\extd x_i,z]=0,\quad [\extd z,x_i]=-\imath\lambda \extd x_i,\quad [\extd z,z]=\imath\lambda \theta.\end{equation}
and $\extd=(\imath\lambda)^{-1}[\theta,\ ]$ as the 4D Poincar\'e-covariant calculus on 3D bicrossproduct spacetime, cf \cite{Sitarz} in other dimensions, but now derived from the bicovariant calculus on $C_q[SU_2]$. Note that $\theta'=\theta+\extd z$ obeys $[\theta',x_i]=0$ and $\theta' z=(z+\imath \lambda)\theta'$ and is therefore more natural to work with than $\theta$ in computations.  \endproof

To complete the picture we define partial derivatives by 
\[ \extd f=\sum_{i=1,2}(\del^if )\extd x_i +(\del^zf) \extd z+(\del^0f)\theta',\quad \theta'=\theta+\extd z\]
as operators on $f$ in the bicrossproduct spacetime algebra. Note that the $\imath z,\imath\extd z,-\theta$ variables in (\ref{bicrosscalc}) form exactly the same algebra as the polar coordinates algebra of $\hat r,\extd \hat r,\theta$ on fuzzy $\R^3$ in \cite{Ma:time,FreMa}, hence we can read off $\extd g(z)$ from formulae there as 1st and 2nd order finite differences. Meanwhile, the $x_i,\extd x_i,\theta'$ algebra behaves very simply as $\theta'$ is central there and has the form of an extended classical calculus
\begin{equation}\label{jetcalc}
\extd f(x,y)=\sum_i {\del f\over \del x_i}\extd x_i+{1\over 2}\sum_i{\del^2 f\over\del x_i^2}\theta'.\end{equation}
Combining these observations immediately gives the result on normal ordered functions $f(x,y)g(z)$ (i.e. keeping the $z$ variable to the right),
\begin{equation}\label{partialbic} \del^i(fg)=({\del f\over \del x_i})g,\quad \del^z(fg)=f\, {g(z)-g(z-\imath\lambda)\over\imath\lambda},\end{equation}
\begin{equation}\label{laplbic}(\imath\lambda)^{-1} \del^0(fg)={1\over 2}\left(\sum_i {\del^2 f\over \del x_i^2}\right)  g(z+\imath\lambda)+{1\over 2}f\, {g(z+\imath\lambda)+g(z-\imath\lambda)-2g(z)\over (\imath\lambda)^2}.\end{equation}
We see that $2(\imath\lambda)^{-1}\del^0$ has the expected classical/finite difference form of the Laplacian on the bicrossproduct spacetime. One can in fact take it as a definition of that.  Its value $-k^2e^{-\omega\lambda}-({\sinh(\omega\lambda/2)\over\lambda/2})^2$ on plane waves $e^{\imath \sum_i k_ix_i}e^{\imath\omega z}$ in bicrossproduct spacetime coincides up to adjustment for the signature with the value previously computed in the noncovariant calculus for bicrossproduct spacetimes, see \cite{AmeMa}.  This emergence of a natural Laplacian is the same phenomenon of `spontaneous evolution' as in \cite{Ma:time} for fuzzy $\R^3$. It is already known that the partial derivative in the $\theta$ direction on $C_q[SU_2]$ gives its Laplace-Beltrami operator and we accordingly define the partial derivative in the $\theta$ direction on $B_q[SU_2]$ as its Laplace-Beltrami operator for the calculus of Proposition~5.

{\bf 5.3} Note finally that from the projector approach of the present paper one can already construct the tautological  line bundle and monopole on all the various deformed spheres  and obtain other vector bundles by tensor product and direct sum. Using the differential calculus one may then compute the explicit form of the monopole connection. However, a  `quantum principal bundle' analogous to the classical Hopf fibration is not necessarily  known. In the case of  the Podle\'s sphere, a suitable
coalgebra bundle \cite{BrzMa:geo} is known with the Podle\'s sphere appearing as an  invariant subalgebra of $C_q[SU_2]$ as  `base' of  this bundle. If this can be naturally adapted to our point of view as a constant-time slice in $B_q[SU_2]$ one may then be able to
take a limit $q\to 1$ and $t\to 1$ with $\lambda$ fixed and thereby obtain a coalgebra bundle on fuzzy-spheres. We may also consider the standard $q$-monopole on $C_q[SU_2]$ as `total space' and consider asymptotically what happens in its limit to bicrossproduct spacetime. Finally, one could use these constructions to construct quantum particle states or representations of the isometry quantum group for the relevant sector of 3D quantum gravity, as explained in \cite{MaSch} via quantum Fourier transform and quantum Born reciprocity. Clearly, a natural calculus on $B_q[SU_2]$ or $U_q(su_2)$ has many possible applications but these are some directions for further work.

\appendix
\section{Construction of exterior algebras by transmutation}

{\bf A.1} In this section we will need a little Hopf algebra theory and refer to \cite{Ma:book} for the methods. The theory works over any field $k$, with $k=\C$ of interest in the body of the paper. Briefly, a Hopf algebra $A$ mean an algebra over a field $k$ equipped with a coproduct $\Delta:A\to A\tens A$, counit $\eps:A\to k$ and antipode $S:A\to A$ obeying some axioms. We let $A^+=\ker\eps$ denote the elements killed by $\eps$. We use the `Sweedler notation' $\Delta a=a\o\tens a\t$. A differential calculus $(\Omega^1,\extd)$ over $A$ is left translation  covariant if the coproduct extends to a well-defined map $\Delta_L: \Omega^1\to A\tens \Omega^1$ by $\Delta_L(a\extd b)=a\o b\o \tens a\t \extd b\t$. It is right-covariant if it extends to a well-defined map $\Delta_R:\Omega^1\to \Omega^1\tens A$ by $\Delta_R(a\extd b)=a\o\extd b\o\tens a\t b\t$, and bicovariant if both. In the latter case there is a natural `minimal' exterior algebra $\Omega=\oplus_i\Omega^i$ due to Woronowicz \cite{Wor:dif} generated by $\Omega^0=A$ and $\Omega^1$ and with $\Omega^i$ for $i>1$ defined by a certain `skew-symmetrization' of $\Omega^1$ with respect to a certain `quantum double' braiding (we assume  for this that $S$ is invertible). There is also a `maximal prolongation' exterior algebra with just the minimal relations implied by applying $\extd$ to the relations at first order. There are also potentially intermediate bicovariant options such as using only the quadratic  degree 2 relations in the Woronowicz construction. It can be shown that  $\Omega$ is a $\Z_2$-graded or `super' Hopf algebra with
\[ \Delta|_{\Omega^0}=\Delta,\quad \Delta|_{\Omega^1}=\Delta_L+\Delta_R\]
extended as a $\Z_2$-graded algebra homomorphism. 

Also in the bicovariant case, one can show that $\Omega=A\Lambda$ in the sense of an algebra factorisation where $\Lambda=\oplus\Lambda^i$ is the subalgebra of (say) left-invariant differential forms. It forms a braided-Hopf algebra with additive coproduct on the generating space $\Lambda^1$, and $\Omega$ is its `super bosonisation'.  In practical terms,  $\Lambda^1$ forms a right $A$-crossed module in the sense of compatible right actions and coactions of $A$ (in such a way as to form a comodule for the quantum codouble of $A$ appropriately defined). The coaction is the restriction $\Delta_R|_{\Lambda^1}$ and we denote it explicitly by $\Delta_R(v)=v\bo\tens v\bt$ for $v\in\Lambda^1$, and the action $\ra$ is defined by the right adjoint action $v\ra b=Sb\o v b\t$ in $\Omega$. One may use the Maurer-Cartan form $\omega:A^+\to \Lambda^1$ defined by $\omega(a)=Sa\o\extd a\t$ to identify $\Lambda^1\cong A^+/I$ where $I=\ker\omega$. In this way bicovariant calculi are in 1-1 correspondence with $\Ad_R$-stable right ideals in $A^+$, where $\Ad_R(a)=a\t\tens (Sa\o)a\th$ using the numerical  notation extended to iterated coproducts.  In these terms the action on $\Lambda^1$ is descended from right multiplication on $A^+$ and the coaction is descended from $\Ad_R$.
 
{\bf A.2} Let $\pi:A'\to A$ be a Hopf algebra map and $A$  coquasitriangular in the sense of a map $\CR:A\tens A\to k$ obeying some standard axioms \cite{Ma:book}. Then the categorical definition of the transmutation \cite{Ma:bg} can be unwound to the explicit result of a new algebra $\underline{A'}$ with modified product
\[ a\bullet b=a\bo b\t \CR(a\bt\tens S\pi(b\o))\]
which, with unchanged coalgebra structure forms `braided group' or Hopf algebra in the braided category $\CM^A$ via the pushed out adjoint coaction $\Ad_\pi(b)=b\t\tens \pi((Sb\o)b\th)$ on $A'$. If $A'$ is a $\Z_2$-graded Hopf algebra then one similarly has $\und{A'}$ a $\Z_2$-graded Hopf algebra in $\CM^A$ with its usual braiding induced by $\CR$ or a Hopf algebra in $\CM^A$ with a modified braiding in which $-1$ factors appear according to the grading. We call such an object a super braided group in the category $\CM^A$.

{\bf A.3} We apply these remarks to $\pi:\Omega\to A$ as a map of super-Hopf algebras in which $A$ has degree 0 (an ordinary Hopf algebra), $\pi|_{\Omega^0}=\id$ and $\pi=0$ on higher degrees. The $\Omega^0=A$ subalgebra transmutes as usual to $\und\Omega^0=\und A$ defined as above with $\pi=\id$. 

\begin{propos} The transmutation of a bicovariant calculus $\Omega$ on coquasitriangular Hopf algebra $(A,\CR)$ is a super braided group $\und\Omega$ which provides a differential graded algebra on $\und A$ (the transmutation of $A$) generated by $\und A$ and $\Lambda$ with relations
\[ \CR(v\bt\tens a\o)v\bo\bullet a\t=a\o\bullet (v\ra a\t),\quad\forall a\in
\und A,\ v\in \Lambda^1,\]
unchanged relations among elements of $\Lambda$ and unchanged differential $\und\extd=\extd$.
\end{propos}
\begin{proof} The iterated coproduct on $\Lambda^1$ is
\[  (\Delta\tens{\id})\Delta v=(\Delta\tens\id)(1\tens v+v\bo\tens v\bt)=1\tens 1\tens v+1\tens v\bo\tens v\bt+v\bo\bo\tens v\bo\bt\tens v\bt\]
and hence the induced coaction on $\Lambda^1$ is
\[ \Ad_\pi(v)=v\bo\tens v\bt\]
since only terms in $A$ in both the outer positions contribute. Hence 
\[v\bullet a=v\bo a\t \CR(v\bt,Sa\o)=a\t (v\bo\ra a\th)\CR(v\bt,Sa\o)\]
while similarly
\[ a\bullet v=a\t v\t\CR((Sa\o) a\th,S\pi(v\o))=a\t v \CR((Sa\o) a\th, S1)=av.\]
Together these provide the commutation relations
\[ v\bullet a= a\t \bullet (v\bo\ra a\th)\CR(v\bt,Sa\o)\]
 between $\und A$ and $\Lambda^1$, which we rearrange as stated using the properties of $\CR$. For degree 2 we similarly find $v\bullet w=v\bo w\t\CR(v\bt,S\pi(w\o))=vw$ is unchanged. Finally, it is elementary to verify that $\extd: \und A\to \und\Omega^1$ intertwines $\Ad_R$ and $\Ad_\pi$ and hence provides a morphism in $\CM^A$, which for the purposes of the appendix we denote by $\und\extd$ (it is the same linear map once the various spaces are identified). It is elementary to verify that $\und\extd$ obeys the Leibniz rule for the transmuted products. \end{proof}

We have given here the explicit form of the relations at the 1-form level. For higher differential forms one may transmute the product of whichever $\Omega$ is used, eg the minimal `Woronowicz exterior calculus'. In the maximal prolongation exterior calculus and hence in any, one has the Maurer-Cartan equations $\extd \omega(a)+\omega(a\o)\omega(a\t)=0$ and we note that $\und\extd a=\extd a=a\o\omega(a\t)=a\o\bullet\omega(a\t)$, where now we consider $\omega: A\to \Lambda^1$. We see that the
Maurer-Cartain form is unchanged under transmutation and obeys the same equations
\[ \und\omega(a)=\und S a\o\bullet \und\extd a\t=(\und S a\o)\bullet a\t\bullet\omega(a\th)=\omega(a),\quad \und\extd\und\omega(a)+\und\omega(a\o)\bullet\und\omega(a\t)=0.\]
Having established these facts we will no longer underline the braided $\extd,\omega$.

{\bf A.4} It remains to derive calculable `R-matrix' formulae for $\und \Omega$ above. We first need to give a clean $R$-matrix presentation in our required left-basis conventions of the initial calculus $\Omega$. For $\Omega(C_q[SU_2])$ such formulae were first found by Jurco, while the systemmatic approach for any coquasitriangular Hopf algebra comes from the classification theorem due to the author\cite{Ma:dcla} using the `quantum killing form' $\CQ:A^+\to H$ where $H^+$ is a suitable dually-paired Hopf algebra to $A$.  Here $\CQ=\CR_{21}\CR\in H\tens H$ is viewed as a map by evaluation against the first factor. We assume for purposes of derivation that $H$ is quasitriangular, however the final formulae are eventually verified directly without strictly assuming this. The ad-invariance properties of the quantum killing form and the properties of $\CR$ imply that $\CQ$ intertwines the $A$-crossed module structure on $A^+$ given by  $\Ad_R$ and right multiplication with the $A$-crossed module structure on $H^+$ given by 
\[ \<\id\tens g,\Delta_Rh\>=g\o h S^{-1}g\t,\quad  \<h\ra a,b\>=\CR(b\o,a\o)\<b\t,h\>\CR(a\t,b\th)\]
for all $g\in H$, $h\in H^+$, and $a,b\in A$. For any standard quantum groups $A=C_q[G]$ associated to a complex semisimple Lie algebra, we take $H=U_q(g)$ and actually work with the composite $Q=\rho\circ \CQ:A^+\to M_n(\C)$ where $\rho$ is the matrix representation of $U_q(g)$ defied by evaluation with the matrix generators $\{t^a{}_b\}$ of $C_q[G]$. Notice that this composite map does not involve any formal powerseries and all its properties may be established directly. The canonical bicovariant calculus on $C_q[G]$ is defined by the $\Ad_R$-stable right ideal $I=\ker Q$. Also,  for generic $q$, the map $Q$ is surjective and provides an isomorphism $\Lambda^1\cong M_n(\C)$. We let $\{e_\alpha{}^\beta\}$ be the standard basis of the latter where $e_\alpha{}^\beta$ has 1 at row $\alpha$ and column $\beta$. Using the above results for $H^+$ and  one may compute that the crossed module structure on $\Lambda^1$  maps over under the $Q$ isomorphism to  
\begin{equation}\label{crossmat} \Delta_Re_\alpha{}^\beta=e_m{}^n\tens t^m{}_\alpha S t^\beta{}_n,\quad e_\alpha{}^\beta\ra t^a{}_b=e_m{}^n R^m{}_\alpha{}^a{}_c R^c{}_b{}^\beta{}_n\end{equation}
where $R^a{}_b{}^m{}_n=\CR(t^a{}_b,t^m{}_n)$ is the `R-matrix'. The differential calculus $\Omega^1(C_q[G])$ is generated by $C_q[G]$ and the $e_\alpha{}^\beta$ with bimodule relations
\[ e_\alpha{}^\beta t^a{}_b=t^a{}_ce_m{}^n R^m{}_\alpha{}^c{}_d R^d{}_b{}^\beta{}_n\]
and the above right covariance. The left covariance is defined by $e_\alpha{}^\beta$ invariant. Finally,  the Maurer-Cartan form is $\omega(a)= Q(a)$ in this description and hence
\[ \extd t^a{}_b=t^a{}_cQ(t^c{}_b-\eps(t^c{}_b))=t^a{}_c (R_{21}R)^c{}_b{}^\alpha{}_\beta e_\alpha{}^\beta-t^a{}_b\theta=[\theta,t^a{}_b],\quad \theta=e_\alpha{}^\alpha\]
using the bimodule relations. The final step is to rescale $\extd$ so as to have a $q\to 1$ limit. 

{\bf A.5} The corresponding transmuted differential calculus follows from Proposition~9 
\[ t^a{}_c\bullet e_\alpha{}^\beta\ra t^c{}_b=\CR((e_\alpha{}^\beta)\bt,t^a{}_c)(e_\alpha{}^\beta)\bo\bullet t^c{}_b\]
which  using the action and coaction (\ref{crossmat}) becomes
\[ \CR(t^m{}_\alpha St^\beta{}_n,t^a{}_c)e_m{}^n\bullet t^c{}_b=R^m{}_\alpha{}^a{}_d R^{-1}{}^\beta{}_n{}^d{}_c e_m{}^n\bullet t^c{}_b=t^a{}_c\bullet e_m{}^n R^m{}_\alpha{}^c{}_d R^d{}_b{}^\beta{}_n.\]

The $\bullet$ relations among the $t^a{}_b$ generators also has an R-matrix form with two R's on each side (subsequently called `reflection equations' by some authors). We denote the elements $t^a{}_b$ of $A=C_q[G]$ when viewed in $\und A=B_q[G]$ by $u^a{}_b$ and henceforth omit the $\bullet$ when working with these as the generators of $B_q[G]$. Their principal relations then take a compact form 
$u_2R_{21}u_1R=R_{21}u_1Ru_2$ as transmutation of the more familiar FRT relations $Rt_1t_2=t_2t_1R$ of $C_q[G]$, see\cite{Ma:book}. We similarly now obtain the relations (\ref{Rcalc}) for the bimodule relations of $\Omega(B_q[G])$ defined by transmutation. They
do not take such a compact form in terms of $R$ (hence we write them with indices) though one can put them in a compact form in terms of a different matrix obtained from $R$.  

Note that the normalisation of $R$ needs to be the `quantum group normalisation'\cite{Ma:book} where $R$ is defined as above from $\CR$. For $C_q[SU_2]$ this is $q^{-{1\over 2}}$ times $R$ in the more standard Hecke normalisation. Also, in general, summing $\alpha=\beta$ in the bimodule relations we see that $\theta u^a{}_b=u^a{}_c e_m{}^n (R_{21}R)^c{}_b{}^m{}_n$ so that $\extd u^a{}_b=[\theta,u^a{}_b]$ by the same computation as before. Hence the transmuted calculus is also inner.  We rescale $\extd$ as before in order to have a classical limit.

{\bf A.6} Although the full picture for the above results comes from transmutation, it is always possible\cite{Ma:euc} to express transmutation as a certain Drinfeld twist, which is adequate if one is only interested in the exterior algebra. Let $A$ be a Hopf algebra with bijective antipode, then $A^{\rm op}$, where we use the opposite product, is also a Hopf algebra and we let $\tilde A=A\tens A^{\rm op}$. It is easy to see that 
\begin{equation}\label{Deltatilde} \Delta_R:A^{\rm op}\to A^{\rm op}\tens A\tens A^{\rm op},\quad a\mapsto a\t\tens Sa\o\tens a\th\end{equation}
makes $A^{\rm op}$ into a right $\tilde A$-comodule algebra. Indeed,
\[ \Delta_R(a\cdot_{op}b)=\Delta_R(ba)=b\t a\t\tens (Sa\o)(S b\o)\tens b\th a\th=a\bo \cdot_{op}b\bo \tens a\bt \tilde\cdot b\bt\]
where we use the product in $\tilde A$ on the right and where we write $\Delta_R=a\bo\tens a\bt$. Now suppose that $A$ is coquasitriangular as above, so equipped with $\CR:A\tens A\to k$, and define $F:\tilde A\tens\tilde A\to k$ by
\begin{equation}\label{FR} F(a\tens b,c\tens d)=\eps(a)\CR^{-1}(b,cd),\quad  F^{-1}(a\tens b, c\tens d)=\eps(a)\CR(b,cd)\end{equation}
in terms of the original structures of $A$. Here the inverse is in the `convolution algebra' and $\CR^{-1}(a, b)=\CR(Sa, b)$. One may verify cf computations in \cite{Ma:book} that this is a dual 2-cocycle on $\tilde A$. 

In this situation the right coaction and 2-cocycle will induce on the vector space of $A^{\rm op}$ a new algebra $A_F$ with product 
\[ a\bullet b=a\bo \cdot_{op} b\bo \,  F^{-1}(a\bt\tens b\bt)=b\t a\t F^{-1}(Sa\o\tens a\th,Sb\o\tens b\th).\]
Moreover, $A_F$ will be a right comodule algebra under a Drinfeld cotwist Hopf algebra $\tilde A^F$, where the latter has a new product
\[ (a\tens b)\cdot_F(c\tens d)=F(a\o\tens b\o,c\o\tens d\o)(a\t\tens b\t)\tilde\cdot (c\t\tens d\t)F^{-1}(a\th\tens b\th,c\th\tens d\th).\]
Note that Drinfeld\cite{Dri:qua} discussed only the twisting of (quasi)-Hopf algebras. The induced twisting and cotwisting of covariant algebras goes back to work of the author based on a categorical point of view. As a result of these cotwists one finds
\[ A^{op}_F=\underline A,\quad \tilde A^F=A\bowtie_\CR A\]
where $\underline A$ is the transmutation as in A.2 above with $\pi=\id$, and $A\bowtie_\CR A$ is a certain double cross product `complexification' factorising into the two copies of $A$ as sub-Hopf algebras, see \cite{Ma:book}. It has product
\[ (a\tens b)(c\tens d)=\CR^{-1}(b\o,c\o)ac\t\tens b\t d\CR(b\th,c\th)\]
and a tensor coproduct structure, and surjects as a Hopf algebra onto $A$ by product of the factors. Its coaction on $\underline{A}$ is the same linear map as the coaction $\Delta_R$ of $\tilde A$. This outlines the results in \cite{Ma:euc} in the conventions we need. 

{\bf A.7} Suppose that $\Omega(A)$ is a bicovariant differential calculus on a Hopf algebra $A$ with bijective antipode. In this case $\Omega(A)^{\rm op}$ with reversed products and unchanged $\extd$ is an exterior algebra on the Hopf algebra $A^{\rm op}$ covariant under the right coaction (\ref{Deltatilde}) of $\tilde A$. Thus
\[ \Delta_R(a\cdot_{op}\extd b)=\Delta_R((\extd b)a)=(\extd b\t)a\t\tens S(b\o a\o)\tens b\th a\th=a\bo\cdot_{op}\extd b\bo\tens a\bt\tilde\cdot b\bt \]
is well-defined (and then necessarily has the required properties). The coaction is the same linear map as the coaction $((\id\tens S)\Delta_L\tens\id)\Delta_R$ of $\tilde A$ on $\Omega^1(A)$ in terms of left and right coations $\Delta_{L,R}$ of $A$. Observe that $\Omega(A^{\rm op}):=\Omega(A)^{\rm op}$ becomes a $\tilde A$-comodule algebra and that the Maurer-Cartan form for $A^{\rm op}$ is
\[ \omega_{op}(a)=S^{-1}a\o\cdot_{op}\extd a\t=(\extd a\t)S^{-1}a\o= -a\t\extd S^{-1}a\o=-\omega(S^{-1}a)\]
so the space of left-invariant 1-forms remains $\Lambda^1$. Note also that the coaction of $\tilde A$ on $v\in\Lambda^1$ is
\begin{equation}\label{Deltatildev}\Delta_R v= v\bo\tens 1\tens v\bt\end{equation}
corresponding to invariance under $\Delta_L$ and with coaction $v\bo\tens v\bt$ of $A$ from the right. We recall that $\Lambda^1$ is also a right module under $A$ and we have denoted this action as $\ra$. These two structures form a right crossed module which we use in the form $\Delta(v\ra a)=v\bo\ra a\t\tens (Sa\o)v\bt a\th$. 

We now apply the induced cotwisting of covariant algebras to $\Omega(A^{\rm op})$ when $A$ is coquasitriangular. This will induce a new product on the exterior algebra. We recall that any coquasitriangular Hopf algebra has a functional $\cv$ implementing the square of the antipode\cite{Ma:bg,Ma:book}. Explicitly, $\cv^{-1}(a)=\CR(S^2a\o,a\t)$. 

\begin{propos} Let $(A,\CR)$ be a coquasitriangular Hopf algebra, $F$ the cocycle (\ref{FR}) and $\Omega(A)$ a bicovariant calculus. The cocycle cotwist $\Omega(A^{\rm op})_F$ provides an exterior algebra on $A_F=\underline{A}$ isomorphic to $\Omega(\underline{A})$ in Proposition~9. It is generated by $\und A,\Lambda^1$, has bimodule relations 
\[ v\bo\bullet a\t\CR(Sv\bt,(Sa\o)a\th)=a\o\bullet (v\bo\ra S^{-1}a\fiv)\cv^{-1}(a\t)\CR(a\th,Sa\six)\CR(a\fo,Sv\bt)\]
and is covariant under $A\bowtie_\CR A$.   \end{propos}
\proof We first compute the products in $\Omega(A^{\rm op})_F$,
\[ a\bullet v=a\t\cdot_{op}v\bo F^{-1}(Sa\o\tens a\th,1\tens v\bt)=v\bo a\o\CR(a\t,v\bt)\]
\[ v\bullet a=v\bo\cdot_{op}a\t F^{-1}(1\tens v\bt,Sa\o\tens a\th)=a\t v\bo\CR(v\bt,(Sa\o)a\th).\]
Both expressions on the right have the form of a braiding in the category $\CM^A$ of right $A$-comodules (for the right regular and right adjoint coactions of $A$ on itself respectively) followed by the original product in $\Omega(A)$. As a result they can both be inverted as
\[  v\bo\bullet a\t \CR^{-1}(v\bt,(Sa\o)a\th)=av,\quad a\o\bullet v\bo\CR^{-1}(a\t,v\bt)=va.\]
The relation $av=(v\ra S^{-1}a\t)a\o$ in $\Omega(A)$ gives the bimodule relation as stated. One may also compute $v\bullet w=w\bo v\bo \CR(v\bt,w\bt)$ again given by the braiding in $\CM^A$ followed by the original product. It remains to prove that this differential calculus is isomorphic to $\Omega(\und A)$.  We outline the proof for the bimodule relations; the general case is similar. First, suppose that $\Theta:\Omega^1(\und{A})\to \Omega^1(A^{\rm op})_F$ is an isomorphism. This means a bimodule map forming a commuting triangle with the respective exterior derivatives. However, in both cases $\extd$ as a linear map is unchanged, so $\Theta(\extd a)=\extd a$. Hence $\Theta(a \bullet\extd b)=a\bullet\extd b=(\extd b\t)a\o\CR(a\t,(Sb\o)b\th)$  for the respective deformed products and the result in terms of $\Omega(A)$. We now take this as a definition of $\Theta$. By construction  it is a left-module map and it is a right module map by the Leibniz rule and the fact that the $\bullet$ products coincide on $A$ as the product of $\und A$. This can also be verified directly from the form of $\Theta$ and repeated use of the properties of a coquasitriangular Hopf algebra. \endproof 
 
In particular, the braided Maurer-Cartan form in $\Omega(\und{A})$ maps over to 
\[ \Theta(\omega(a))=\Theta(\und{S}a\o\extd a\t)=\und{S}a\o\bullet \extd a\t=-\cv^{-1}(a\o)\CR(a\t,Sa\fo)\omega(S^{-1}a\th)\] after some computation, so $\Theta$ is not the identity map when restricted to $\Lambda^1$. We  view the right hand side as a cotwisted version of $\omega_{op}$ and note that $\extd a=a\o\bullet\Theta(\omega(a\t))$ in $\Omega(A^{\rm op})_F$. 

\begin{corollary} The calculus $\Omega(\und{A})$ is covariant under $\tilde A=A\tens A^{\rm op}$ with coactions (\ref{Deltatilde}) and (\ref{Deltatildev}) on $A,\Lambda^1\subset \Omega(\und A)$.\end{corollary}
\proof We check that $\Theta$ on $\Lambda^1$, although not the identity, commutes with the coaction (\ref{Deltatildev}). Indeed, $(\Theta\tens\id)\Delta_R \omega(a)=\Delta_R\Theta(\omega(a))$ since both sides compute as
\[ -\cv^{-1}(a\o)\CR(a\th, Sa\fiv)\omega(S^{-1}a\fo)\tens 1\tens (S^{-1}a\t)a\six\]
using the definitions of the various maps and identities in \cite{Ma:book} for coquasitriangular Hopf algebras. It is also clear from general principles that $\Theta:\Lambda^1\to \Lambda^1$ is an isomorphism of linear maps. Indeed, its inverse is $\Theta^{-1}(\omega(a))=-\omega(Sa\t)\cu(a\th)\CR(a\fo,a\o)$ where $\cu(a)=\CR(a\t,Sa\o)$ also implements the square of the antipode\cite{Ma:bg,Ma:book}. \endproof

Note that these coactions, when pushed out along the product map $A\bowtie_\CR A\to A$ become our original coactions by which $\Omega(\und A)$ was an object in the braided category $\CM^A$ in Proposition~9; we have succeeded in lifting them to coactons of $A\bowtie_\CR A$, which is also a coquasitriangular Hopf algebra. For $A$ factorisable, which is essentially the case for the standard quantum groups $C_q[G]$,  $C_q[G]\bowtie_R C_q[G]$ is essentially a version of its quantum double, but in a form which is a deformation of a commutative Hopf algebra and hence which can be viewed as $C_q[G_\C]$, where $G_\C$ is the complexification of $G$ as a real Lie group, see \cite{Ma:book}.

\end{document}